\documentclass[11pt,a4paper]{amsart}
\usepackage[T1]{fontenc}    
\usepackage[utf8]{inputenc}
\usepackage{enumerate}
\usepackage{fullpage}
\usepackage{color}

\usepackage{mathtools}

\allowdisplaybreaks

\newcommand{\set}[1]{\left\lbrace #1\right\rbrace}
\newcommand{\abs}[1]{\left\lvert#1\right\rvert}
\newcommand{\qtq}[1]{\quad \text{#1}\quad }
\newcommand{\bj}[1]{\left(#1\right)}

\newtheorem{theorem}{Theorem}[section]

\newtheorem{lemma}[theorem]{Lemma}
\newtheorem{corollary}[theorem]{Corollary}

\theoremstyle{remark}
\newtheorem{remark}[theorem]{Remark}
\newtheorem*{remark*}{Remark}
\newtheorem{remarks}[theorem]{Remarks}
\newtheorem{example}[theorem]{Example}

\def\NN{\mathbb N}
\def\eps{\varepsilon }

\numberwithin{equation}{section}

\begin{document}

\title{Estimates of Lucas series and expansions in non-integer bases}

\author[V. Komornik]{Vilmos Komornik}
\address{Département de mathématique,
Université de Strasbourg,
7 rue René Descartes,
67084 Strasbourg Cedex, France}
\email{vilmos.komornik@math.unistra.fr}

\author[M. Pedicini]{Marco Pedicini}
\address{Università Roma Tre,
Dipartimento di Matematica e Fisica,
Largo S. Leonardo Murialdo n. 1, 00146 Roma RM, Italy}
\email{marco.pedicini@uniroma3.it}

\date{Version of 2026-06-27a}

\begin{abstract}
Special Fibonacci and Lucas type expansions proved  useful to ``correct'' some theorems on the number of expansions in non-integer bases \cite{BaiKomLor2021, KomLorPed2024, KomLorPed2025}.
The proof of a key lemma was non-constructive.
It was made constructive in the Fibonacci case by giving a precise estimate of the growth of the partial sums of the series of the inverses of the Fibonacci numbers \cite{KomPed2025}.
In this paper we extend this result to the general case.
\end{abstract}
\maketitle

\section{Introduction}\label{s1}

In integer bases every real number has at most two  expansions.
In non-integer bases the situation is completely different: if $M$ is a positive integer and $M<q<M+1$, then Lebesgue \emph{almost every} $x\in [0,\frac{M}{q-1}]$ has $2^{\aleph_0}$ expansions of the form
\begin{equation}\label{11}
x=\sum_{i=1}^{\infty}\frac{c_i}{q^i}, \quad (c_i)\in\set{0,1,\ldots,M}^{\NN}
\end{equation} 
\cite{Sid2003, DajDeV2007}.
Moreover, there exists a largest base $q_M$ such that if $q<q_M$, then \emph{every} $x\in (0,\frac{M}{q-1})$ has $2^{\aleph_0}$ expansions of the form \eqref{11} \cite{EggVanEyn1966, ErdJooKom1990, ErdHorJoo1991, Bak2014}.
For example, $q_1$ is the Golden ratio.
More generally, $q_{2p-1}=\bj{p+\sqrt{p^2+4p}}/2$ and $q_{2p}=p+1$ for $p=1,2,\ldots .$
Similar results hold for arbitrary ternary alphabets, too \cite{KomLaiPed2011,KomPed2017,Ste2020}.

If $M=2p-1$ is odd, then the last property no longer holds for $q=q_M$.
However, the property still holds for a closely related expansion with integer denominators.
Indeed, introducing the Lucas sequence
\begin{equation}\label{12}
F_0=0,\quad F_1=1,\qtq{and} F_k=p(F_{k-1}+F_{k-2})\qtq{for}k=2,3,\ldots
\end{equation}
(for simplicity we do not indicate their dependence on $p$), we have $F_{k+1}/F_k\to q_M$ as $k\to\infty$, and every $x\in (0,\sum_{i=1}^{\infty}\frac{M}{F_i})$ has $2^{\aleph_0}$ expansions of the form
\begin{equation*}
x=\sum_{i=1}^{\infty}\frac{c_i}{F_i}, \quad (c_i)\in\set{0,1,\ldots,M}^{\NN}
\end{equation*}
\cite{BaiKomLor2021, KomLorPed2024, KomLorPed2025}.

The convergence relation $F_{k+1}/F_k\to q_M$ easily follows from the explicit formula
\begin{equation*}
F_k=\frac{A^k-B^k}{A-B},\quad k=0,1,\ldots,
\end{equation*}
where 
\begin{equation*}
A:=q_M=\frac{p+\sqrt{p^2+4p}}{2}\qtq{and}B:=\frac{p-\sqrt{p^2+4p}}{2}=-\frac{p}{A}
\end{equation*}
are the solutions of the equation $x^2=px+p$, and from the obvious inequalities
\begin{equation*}
p<A<p+1\qtq{and}-1<B<0.
\end{equation*}

A main idea of the proof in \cite{BaiKomLor2021} was that in case $M=1$ we may remove from the sequence $(1/F_i)$ infinitely many elements so that the remaining sequence $(p_i)$ satisfies the inequalities
\begin{equation*}
p_n\le\sum_{i=n+1}^{\infty}p_i
\end{equation*}
for every $n$. 
The proof of this lemma was not constructive.
The following more precise estimates, obtained in \cite{KomPed2025}, imply that the removal of the elements $1/F_{2\cdot 3^n}$ from $(1/F_i)$ yields such a sequence $(p_n)$:

\begin{theorem}\label{t11}
Let $(F_i)$ denote the Fibonacci sequence.
If $k$ is a positive odd integer, then
\begin{equation*}
\sum_{i=k+2}^{3k+1}\frac{1}{F_i}<\frac{1}{F_k}<\sum_{i=k+2}^{3k+2}\frac{1}{F_i}.
\end{equation*}
\end{theorem}

The object of the present work is to extend Theorem \ref{t11}  to the Lucas sequences \eqref{12}; they similarly lead to a constructive proof of the main results of \cite{KomLorPed2024, KomLorPed2025}.

The method of our proof allows us to consider a larger class of Lucas sequences.
We fix two integers $p\ge q\ge 1$, and we define the sequence
\begin{equation}\label{13}
F_0=0,\quad F_1=1,\qtq{and} F_k=pF_{k-1}+qF_{k-2}\qtq{for}k=2,3,\ldots ;
\end{equation}
for simplicity we do not indicate their dependence on $p,q$.

Binet's formula readily extends to these sequences: introducing the solutions
\begin{equation*}
A:=\frac{p+\sqrt{p^2+4q}}{2}\qtq{and}B:=\frac{p-\sqrt{p^2+4q}}{2}=-\frac{q}{A}
\end{equation*}
of the equation $x^2=px+q$, we have $p<A<p+1$, $-1<B<0$, and
\begin{equation*}
F_k=\frac{A^k-B^k}{A-B}
\end{equation*}
for all $k\ge 0$.

First we consider the special case where $q=1$.

\begin{theorem}\label{t12}
If $p\ge 2$ and $q=1$, then
\begin{equation}\label{14}
\frac{p-1}{F_{k+1}}+\sum_{i=k+2}^{3k}\frac{p}{F_i}
<\frac{1}{F_k}
<\frac{p-1}{F_{k+1}}+\sum_{i=k+2}^{3k+1}\frac{p}{F_i}
\end{equation}
for all positive odd integers $k$.
\end{theorem}

Observe the difference between the cases $p=1$ and $p\ge 2$: the summation limits are $3k+1$ and $3k+2$ in Theorem \ref{t11}, and $3k$ and $3k+1$ in Theorem \ref{t12}.

Our next result shows that the summation limits are more complicated for $q\ge 2$.

\begin{theorem}\label{t13}
If $p\ge q\ge 2$, then we have
\begin{equation}\label{15}
\frac{p-1}{F_{k+1}}+\sum_{i=k+2}^{K_1(k)}\frac{p+q-1}{F_i}
<\frac{1}{F_k}
<\frac{p-1}{F_{k+1}}+\sum_{i=k+2}^{K_2(k)}\frac{p+q-1}{F_i}
\end{equation}
for all positive odd integers $k$,
where
\begin{equation*}
K_1(k)=
\left\lfloor\bj{3-\frac{\ln q}{\ln A}}k+\frac{\ln(p+q-1)}{\ln A}-1\right\rfloor
\end{equation*}
and
\begin{equation*}
K_2(k)=
\left\lceil\bj{3-\frac{\ln q}{\ln A}}k+\frac{\ln(p+q-1)-2\ln q}{\ln A}+1\right\rceil.
\end{equation*}
\end{theorem}

\begin{remark}\label{r14}
It follows from the definition that
\begin{align*}
K_1(k)&=\max\left\{K\in\mathbb Z: q^kA^{K-3k+1}\le p+q-1\right\}
\intertext{and}
K_2(k)&=\min\left\{K\in\mathbb Z: q^{k+2}A^{K-3k-1}\ge p+q-1\right\}.
\end{align*}
\end{remark}

\begin{example}\label{e15}
For $p=q=2$ and $k=1$ we have $K_1(k)=2$ and $K_2(k)=4$ by a straightforward computation.
Since $F_1=1$, $F_2=2$, $F_3=6$ and $F_4=16$,
 \eqref{15} takes the form
\begin{equation*}
\frac{1}{2}<1<\frac{1}{2}+\frac{3}{6}+\frac{3}{16};
\end{equation*}
its validity is obvious. 
\end{example}

\begin{remarks}\label{r16}\mbox{}
\begin{enumerate}[\upshape (i)]
\item The conclusion of Theorem \ref{t13} also holds if $p\ge 2$ and $q=1$, but this conclusion is weaker than Theorem \ref{t12}: $K_1(k)=3k-1$ and $K_2(k)=3k+2$,   instead of $3k$ and $3k+1$, respectively.

\item Note that $K_2(k)>K_1(k)\ge k+1$, and hence $K_2(k)\ge k+2$.
Indeed, the inequality $K_2(k)>K_1(k)$ is obvious, and the next one follows by using the relations $1\le q\le p<A$:
\begin{equation*}
\bj{2-\frac{\ln q}{\ln A}}k+\frac{\ln(p+q-1)}{\ln A}-1
\ge \bj{2-\frac{\ln q}{\ln A}}+\frac{\ln(p+q-1)}{\ln A}-1
\ge 2-1=1.
\end{equation*}
\end{enumerate}
\end{remarks}

\begin{remarks}\label{r17}\mbox{}
Let us focus on the special case $p=q\ge 2$.

\begin{enumerate}[\upshape (i)]
\item We have
\begin{equation*}
K_2(k)-K_1(k)\le 4-\frac{2\ln q}{\ln A}<3
\end{equation*}
because $q^2=p^2>p+1>A$.
Hence $K_2(k)=K_1(k)+1$ or $K_2(k)=K_1(k)+2$ for every odd integer $k\ge 1$, so that Theorem \ref{t13} is essentially optimal.

\item We consider the case $p=q=2$.
Table \ref{tab1} shows the values of 
\begin{equation*}
S_K:=\frac{p-1}{F_{k+1}}+\bj{\sum_{i=k+2}^{K}\frac{p+q-1}{F_i}}-\frac{1}{F_k}=\frac{1}{F_{k+1}}+\bj{\sum_{i=k+2}^{K}\frac{3}{F_i}}-\frac{1}{F_k}
\end{equation*}
for $K=K_1(k), K_1(k)+1, K_1(k)+2$ for those values of $k=1,3,\ldots,49$ for which $K_2(k)=K_1(k)+2$.

We see that in Theorem \ref{t13} sometimes, but not always, we may change $K_1(k)$ to $K_1(k)+1$, and sometimes, but not always, we may change $K_2(k)$ to $K_2(k)-1$.

\begin{table}[ht]
\begin{equation*}
\begin{array}{cccc}
k & S_{K_1(k)} & S_{K_1(k)+1} & S_{K_1(k)+2} \\
\hline\\
 1 & -0.5 & 0 & 0.1875 \\
 5 & -0.0000611432 & -1.04706\times 10^{-6} & 0.0000209497 \\
 9 & -8.06104\times 10^{-9} & -9.7428\times 10^{-10} & 1.61965\times 10^{-9} \\
 11 & -4.34549\times 10^{-11} & 3.10409\times 10^{-12} & 2.01459\times 10^{-11} \\
 15 & -5.88597\times 10^{-15} & -3.95559\times 10^{-16} & 1.61407\times 10^{-15} \\
 19 & -7.64562\times 10^{-19} & -1.17113\times 10^{-19} & 1.1987\times 10^{-19} \\
 21 & -4.24039\times 10^{-21} & 1.32618\times 10^{-23} & 1.57021\times 10^{-21} \\
 25 & -5.62757\times 10^{-25} & -6.11506\times 10^{-26} & 1.2245\times 10^{-25} \\
 27 & -3.00067\times 10^{-27} & 2.94819\times 10^{-28} & 1.50105\times 10^{-27} \\
 31 & -4.09664\times 10^{-31} & -2.10487\times 10^{-32} & 1.21194\times 10^{-31} \\
 35 & -5.34734\times 10^{-35} & -7.64636\times 10^{-36} & 9.12748\times 10^{-36} \\
 37 & -2.93954\times 10^{-37} & 7.12342\times 10^{-39} & 1.17325\times 10^{-37} \\
 41 & -3.92602\times 10^{-41} & -3.75605\times 10^{-42} & 9.23936\times 10^{-42} \\
 43 & -2.06885\times 10^{-43} & 2.63718\times 10^{-44} & 1.1175\times 10^{-43} \\
 47 & -2.84871\times 10^{-47} & -9.80565\times 10^{-49} & 9.08752\times 10^{-48} \\
\end{array}
\end{equation*}
\caption{Non-optimality of $K_1(k)$ or $K_2(k)$ for $k$'s with $K_2(k)=K_1(k)+2$ for $p=q=2$}
\label{tab1}
\end{table}

\item Let us write the formulas of $K_1(k)$ and $K_2(k)$ in the form
\begin{equation*}
K_1(k)=\left\lfloor \bj{3-\frac{\ln q}{\ln A}}k+\alpha_{p,q}\right\rfloor
\qtq{and}
K_2(k)=\left\lceil \left(3-\frac{\ln q}{\ln A}\right)k+\beta_{p,q}\right\rceil
\end{equation*}
for brevity.
Since
\begin{equation*}
p+q-1>A>1,\quad
\beta_{p,q}-\alpha_{p,q}=2-2\frac{\ln q}{\ln A}\qtq{and}q^2>A
\end{equation*} 
(see (i)), we have $0<\alpha_{p,q}<\beta_{p,q}<1$.
Therefore $K_2(k)=K_1(k)+2$ if and only if there exists an integer $n$ such that
\begin{equation*}
\bj{3-\frac{\ln q}{\ln A}}k+\alpha_{p,q}<n<\left(3-\frac{\ln q}{\ln A}\right)k+\beta_{p,q}.
\end{equation*}
We may rewrite this in the equivalent form
\begin{equation*}
\bj{3-\frac{\ln q}{\ln A}}k\in\bj{-\beta_{p,q},-\alpha_{p,q}}\quad\mod 1.
\end{equation*}
By a classical equidistribution theorem 
the density of the odd integers $k$ (among all positive odd integers) for which $K_2(k)=K_1(k)+2$, is equal to $\beta_{p,q}-\alpha_{p,q}=2-2\frac{\ln q}{\ln A}$.

\item Since $p=q$ and $p<A<p+1$, 
\begin{equation*}
\beta_{p,q}-\alpha_{p,q}=2-2\frac{\ln p}{\ln A}\to 0\qtq{as}p\to\infty,
\end{equation*}
so that the frequency of $k$'s for which Theorem \ref{t13} does not give the optimal values of $K_1(k)$ and $K_2(k)$, tends to zero as $p\to \infty$.
See Table \ref{tab2}.

\begin{table}[ht]
\begin{equation*}
\begin{array}{cccccc}
p=q & \text{density} & p=q & \text{density} & p=q & \text{density} \\
\hline\\
 2 & 0.6206747332 & 100 & 0.004270075079 & 200 &
   0.001871638313 \\
 3 & 0.3513055630 & 101 & 0.004219397806 & 201 &
   0.001860653889 \\
 4 & 0.2390904081 & 102 & 0.004169824553 & 202 &
   0.001849789198 \\
 5 & 0.1784855318 & 103 & 0.004121320646 & 203 &
   0.001839042345 \\
 6 & 0.1409409855 & 104 & 0.004073852824 & 204 &
   0.001828411472 \\
 7 & 0.1155940080 & 105 & 0.004027389174 & 205 &
   0.001817894761 \\
 8 & 0.09743655298 & 106 & 0.003981899055 & 206 &
   0.001807490430 \\
 9 & 0.08385101207 & 107 & 0.003937353046 & 207 &
   0.001797196735 \\
 10 & 0.07334153215 & 108 & 0.003893722879 & 208 &
   0.001787011964 \\
 11 & 0.06499406526 & 109 & 0.003850981385 & 209 &
   0.001776934444 \\
 12 & 0.05822008702 & 110 & 0.003809102447 & 210 &
   0.001766962533 \\
\end{array}
\end{equation*}
\caption{Density of $k$'s for which $K_2(k)=K_1(k)+2$}
\label{tab2}
\end{table}
\end{enumerate}
\end{remarks}

We end the paper with an application of Theorem \ref{t13} to expansions in non-integer bases.
Fix a positive integer $p$, and consider the Fibonacci or Lucas sequence associated with $p=q$.
We are investigating the expansions of the form
\begin{equation}\label{16}
x=\sum_{i=1}^{\infty}\frac{c_i}{F_i}\qtq{with}(c_i)\in\set{0,1,\ldots,2p-1}^{\NN}.
\end{equation}
The following theorem was obtained in \cite{BaiKomLor2021} for $p=q=1$ and in \cite{KomLorPed2024,KomLorPed2025} for $p=q\ge 2$:

\begin{theorem}\label{t18}
If $p=q\ge 1$, then every $x\in\bj{0,\sum_{i=1}^{\infty}\frac{2p-1}{F_i}}$ has $2^{\aleph_0}$ expansions of the form \eqref{16}.
\end{theorem} 

The proof of Theorem \ref{t18} was based on the following result: \emph{there exists an infinite sequence of indices $1\le n_1<n_2<\cdots$ such that}
\begin{equation}\label{17}
\frac{1}{F_{n_j}}<\frac{p-1}{F_{n_j+1}}+\sum_{i=n_j+2}^{n_{j+1}}\frac{2p-1}{F_i},\quad j=1,2,\ldots ;
\end{equation}
see the proof of \cite[Theorem 1.11]{KomLorPed2025}.
While in \cite{BaiKomLor2021,KomLorPed2024,KomLorPed2025} only the existence of such sequences $(n_j)$ was proved, 
Theorems \ref{t11} and \ref{t13} readily imply an explicit construction:

\begin{corollary}\label{c19}\
\begin{enumerate}[\upshape (i)]
\item If $p=q=1$, then the estimates \eqref{17} are satisfied with $n_j:= 2\cdot 3^j-1$.
\item If $p=q\ge 2$, then the estimates \eqref{17} are satisfied with the sequence $(n_j)$ defined by the following recursion, starting with $n_1=1$.
If $j\ge 2$, and $n_{j-1}$ has also been defined, then let $n_j$ be the smallest odd integer satisfying $n_j\ge K_2(n_{j-1})$.
\end{enumerate}
\end{corollary}

\begin{remark}\label{r110}
In view of the explicit formula of $K_2(k)$ in Theorem \ref{t13} the following asymptotic relations hold:
\begin{equation*}
n_j\sim \bj{3-\frac{\ln p}{\ln A}}^{j-1},\quad j\to\infty.
\end{equation*}
Furthermore, 
\begin{equation*}
3-\frac{\ln p}{\ln A}\to 2\qtq{as}p\to\infty.
\end{equation*}
\end{remark}

We prove Theorem \ref{t12} in Sections \ref{s2}--\ref{s3}, and then Theorem \ref{t13} in Sections \ref{s4}--\ref{s5}.
The proof of Theorem \ref{t13} is more technical, and it will use a crucial idea from the proof of Theorem \ref{t12}.

\section{Proof of the lower estimate of Theorem \ref{t12}}\label{s2}

We fix $p\ge 2$ and $q=1$.
It follows from the definition that
\begin{equation}\label{21}
0=F_0<F_1<\cdots,\qtq{and}F_n\ge pF_{n-1}\qtq{for every}n\ge 1.
\end{equation}
We need four lemmas.
The first one is well-known; we recall a short proof for convenience:

\begin{lemma}\label{l21}
The following identities hold:
\begin{align}
&F_{i+j}=F_{i-1}F_j+F_iF_{j+1}\qtq{for all} i\ge 1\qtq{and} j\ge 0;\label{22}\\
&F_iF_{j+1}-F_{i+1}F_j=(-1)^jF_{i-j}\qtq{for all} i\ge j\ge 0.\label{23}
\end{align}
\end{lemma}

\begin{proof}
The identity \eqref{22} reduces to the obvious relations $F_i=F_i$ and $F_{i+1}=F_{i-1}+pF_i$ for $j=0$ and $j=1$, and the general case follows by induction on $j$:
\begin{align*}
F_{i+j+1}
&=pF_{i+j}+F_{i+j-1}
=p\bj{F_{i-1}F_j+F_iF_{j+1}}+\bj{F_{i-1}F_{j-1}+F_iF_j}\\
&=F_{i-1}\bj{pF_j+F_{j-1}}+F_i\bj{pF_{j+1}+F_j}
=F_{i-1}F_{j+1}+F_iF_{j+2}.
\end{align*}

Since the identity \eqref{23} reduces to the obvious equality $F_i=F_i$ for $j=0$, it is sufficient to show that the left hand side only changes sign when we replace $(i,j)$ by $(i+1,j+1)$.
This is shown as follows:
\begin{align*}
\bj{F_{i+1}F_{j+2}-F_{i+2}F_{j+1}}+\bj{F_iF_{j+1}-F_{i+1}F_j}
&=F_{i+1}\bj{F_{j+2}-F_j}-F_{j+1}\bj{F_{i+2}-F_i}\\
&=F_{i+1}pF_{j+1}-F_{j+1}pF_{i+1}=0.\qedhere
\end{align*}
\end{proof}

\begin{lemma}\label{l22}
The following identity holds for all $i\ge 2$:
\begin{equation*}
\frac{p}{F_i}=\frac{1}{F_{i-1}}-\frac{1}{F_{i+1}}+(-1)^i\frac{p}{F_{i-1}F_iF_{i+1}}.
\end{equation*}
\end{lemma}

\begin{proof}
Multiplying by the denominators we rewrite the identity in the form
\begin{equation*}
pF_{i-1}F_{i+1}=F_iF_{i+1}-F_{i-1}F_i+(-1)^ip.
\end{equation*}
Since $pF_{i+1}+F_i=F_{i+2}$, this is equivalent to
\begin{equation*}
F_{i-1}F_{i+2}=F_iF_{i+1}+(-1)^ip.
\end{equation*}
The last identity follows from the identity \eqref{23}:
\begin{equation*}
F_{i+1}F_i-F_{i+2}F_{i-1}=(-1)^{i-1}F_2
\end{equation*}
because $F_2=p$.
\end{proof}

%

\begin{lemma}\label{l23}
If $k\ge 1$ is odd, then
\begin{equation}\label{24}
\bj{\frac{p-1}{F_{k+1}}+\sum_{i=k+2}^{3k}\frac{p}{F_i}}-\frac{1}{F_k}
< \frac{p}{F_kF_{k+1}F_{k+2}}-\frac{1}{F_{3k}}-\frac{1}{F_{3k+1}}.
\end{equation}
\end{lemma}

\begin{proof}
Using Lemma \ref{l22}, and  simplifying a resulting telescopic sum, we obtain that
\begin{align*}
\bj{\frac{p-1}{F_{k+1}}+\sum_{i=k+2}^{3k}\frac{p}{F_i}}-\frac{1}{F_k}
&=-\frac{1}{F_k}-\frac{1}{F_{k+1}}+\sum_{i=k+1}^{3k}\frac{p}{F_i}\\
&=-\frac{1}{F_k}-\frac{1}{F_{k+1}}+\sum_{i=k+1}^{3k}\bj{\frac{1}{F_{i-1}}-\frac{1}{F_{i+1}}+(-1)^i\frac{p}{F_{i-1}F_iF_{i+1}}}\\
&=-\frac{1}{F_{3k}}-\frac{1}{F_{3k+1}}+\sum_{i=k+1}^{3k}(-1)^i\frac{p}{F_{i-1}F_iF_{i+1}}.
\end{align*}
The last sum is a Leibniz type alternating sum  by \eqref{21}, and it is majorized by its first term because it is positive.
This completes the proof.
\end{proof}

For the lower estimate of Theorem \ref{t12} it is sufficient to show that the right hand side of the inequality \eqref{24} is negative.
For this we need another lemma:

\begin{lemma}\label{l24}
The following identity holds for all $k\ge 1$:
\begin{equation}\label{25}
F_{3k}=F_k\bj{F_{k-1}^2+F_k^2+F_{k-1}F_{k+1}+F_{k+1}^2}.
\end{equation}
\end{lemma}

\begin{proof}
Using the identity \eqref{22} we have the following equalities:
\begin{align*}
&F_{3k}=F_{2k-1}F_k+F_{2k}F_{k+1},\\
&F_{2k-1}=F_{k-1}^2+F_k^2,\\
&F_{2k}=F_k\bj{F_{k-1}+F_{k+1}}.
\end{align*}
Combining them we get \eqref{25}.
\end{proof}

\begin{proof}[Proof of the lower estimate of Theorem \ref{t12}]
As mentioned above, it is sufficient to establish the inequality
\begin{equation*}
\frac{p}{F_kF_{k+1}F_{k+2}}-\frac{1}{F_{3k}}-\frac{1}{F_{3k+1}}<0.
\end{equation*}
Since $F_{3k+1}\le (p+1)F_{3k}$ by \eqref{13} and \eqref{21},
it is sufficient to establish the stronger estimate
\begin{equation*}
\frac{p}{F_kF_{k+1}F_{k+2}}-\frac{1}{F_{3k}}-\frac{1}{(p+1)F_{3k}}<0.
\end{equation*}
Multiplying by $(p+1)F_{3k}F_kF_{k+1}F_{k+2}$, this is equivalent to
\begin{equation*}
L:=(p+2)F_{k+1}F_{k+2}-p(p+1)\frac{F_{3k}}{F_k}>0.
\end{equation*}
Applying Lemma \ref{l24} we obtain that
\begin{equation*}
L=(p+2)F_{k+1}F_{k+2}-p(p+1)\bj{F_{k-1}^2+F_k^2+F_{k-1}F_{k+1}+F_{k+1}^2}.
\end{equation*}
Using the identity $F_{k+2}=pF_{k+1}+F_k$ we eliminate $F_{k+2}$ to get
\begin{equation*}
L=pF_{k+1}^2+(p+2)F_kF_{k+1}-p(p+1)\bj{F_{k-1}^2+F_k^2+F_{k-1}F_{k+1}}.
\end{equation*}
Using the inequality $F_{n+1}\ge pF_n$ first with $n=k-1$, and later with $n=k$, we obtain that
\begin{align*}
L&\ge pF_{k+1}^2+(p+2)F_kF_{k+1}-p(p+1)\bj{p^{-2}F_k^2+F_k^2+p^{-1}F_kF_{k+1}}\\
&=pF_{k+1}^2+F_kF_{k+1}-p(p+1)\bj{p^{-2}+1}F_k^2\\
&\ge \left[p^3+p-p(p+1)\bj{p^{-2}+1}\right]F_k^2\\
&=p^3\left[1-\frac{1}{p}-\frac{1}{p^3}-\frac{1}{p^4}\right]F_k^2.
\end{align*}
Since $p\ge 2$, all factors of the last product are positive, and hence $L>0$, as required.
\end{proof}

\section{Proof of the upper estimate of Theorem \ref{t12}}\label{s3}

As in the preceding section, we fix $p\ge 2$ and $q=1$.

\begin{lemma}\label{l31}
If $k\ge 1$ is odd, then
\begin{equation}\label{31}
\bj{\frac{p-1}{F_{k+1}}+\sum_{i=k+2}^{3k+1}\frac{p}{F_i}}-\frac{1}{F_k}
>\frac{p}{F_kF_{k+1}F_{k+2}}-\frac{p}{F_{k+1}F_{k+2}F_{k+3}}-\frac{1}{F_{3k+1}}-\frac{1}{F_{3k+2}}.
\end{equation}
\end{lemma}

\begin{proof}
Similarly to the proof of Lemma \ref{l23}, now we obtain that
\begin{align*}
\bj{\frac{p-1}{F_{k+1}}+\sum_{i=k+2}^{3k+1}\frac{p}{F_i}}-\frac{1}{F_k}
&=-\frac{1}{F_k}-\frac{1}{F_{k+1}}+\sum_{i=k+1}^{3k+1}\frac{p}{F_i}\\
&=-\frac{1}{F_k}-\frac{1}{F_{k+1}}+\sum_{i=k+1}^{3k+1}\bj{\frac{1}{F_{i-1}}-\frac{1}{F_{i+1}}+(-1)^i\frac{p}{F_{i-1}F_iF_{i+1}}}\\
&=-\frac{1}{F_{3k+1}}-\frac{1}{F_{3k+2}}+\sum_{i=k+1}^{3k+1}(-1)^i\frac{p}{F_{i-1}F_iF_{i+1}}.
\end{align*}
Now the last sum is \emph{minorized} by the sum of its first two terms, and this yields the estimate \eqref{31}.
\end{proof}

For the upper estimate of Theorem \ref{t12} it is sufficient to show that the right hand side of the inequality \eqref{31} is \emph{positive}.

We need two further identities:

\begin{lemma}\label{l32}
The following identities hold for all $k\ge 0$:
\begin{align}
F_{3k+1}&=F_{k+1}^3+3F_k^2F_{k+1}-pF_k^3;\label{32}\\
F_{3k+2}&=pF_{k+1}^3+3F_kF_{k+1}^2+F_k^3.\label{33}
\end{align}
\end{lemma}

\begin{proof}
For $k=0$ the identities reduce to the obvious equalities $1=1$ and $p=p$.
Henceforth we assume that $k\ge 1$.
Using the identities \eqref{13} and \eqref{22} we have
\begin{align*}
F_{3k+1}
&=F_{2k+(k+1)}=F_{2k-1}F_{k+1}+F_{2k}F_{k+2}\\
&=\bj{F_{k-1}^2+F_k^2}F_{k+1}+F_k\bj{F_{k-1}+F_{k+1}}F_{k+2}\\
&=\bj{F_{k+1}-pF_k}^2F_{k+1}+F_k^2F_{k+1}+F_kF_{k+2}\bj{2F_{k+1}-pF_k}\\
&=F_{k+1}^3-2pF_kF_{k+1}^2+\bj{p^2+1}F_k^2F_{k+1}+F_k\bj{pF_{k+1}+F_k}\bj{2F_{k+1}-pF_k}\\
&=F_{k+1}^3-2pF_kF_{k+1}^2+\bj{p^2+1}F_k^2F_{k+1}+2pF_kF_{k+1}^2+\bj{2-p^2}F_k^2F_{k+1}-pF_k^3\\
&=F_{k+1}^3+3F_k^2F_{k+1}-pF_k^3.
\end{align*}
Next, using also \eqref{32} and \eqref{25}, we have
\begin{align*}
F_{3k+2}
&=pF_{3k+1}+F_{3k}\\
&=p\bj{F_{k+1}^3+3F_k^2F_{k+1}-pF_k^3}
+F_k\bj{F_{k-1}^2+F_k^2+F_{k-1}F_{k+1}+F_{k+1}^2}\\
&=p\bj{F_{k+1}^3+3F_k^2F_{k+1}-pF_k^3}
+F_k\bj{\bj{F_{k+1}-pF_k}^2+F_k^2+\bj{F_{k+1}-pF_k}F_{k+1}+F_{k+1}^2}\\
&=pF_{k+1}^3+3F_kF_{k+1}^2+F_k^3.\qedhere
\end{align*}
\end{proof}

\begin{proof}[Proof of the upper estimate of Theorem \ref{t12}]
By Lemma \ref{l31} it is sufficient to prove that the right hand side of the inequality \eqref{31} is positive.
This is equivalent to the relation
\begin{equation*}
R:=\frac{F_kF_{k+1}F_{k+2}}{F_{3k+1}}+\frac{F_kF_{k+1}F_{k+2}}{F_{3k+2}}+\frac{pF_k}{F_{k+3}}<p.
\end{equation*}
Using the relations \eqref{13}, \eqref{21}, \eqref{32} and \eqref{33} we may estimate $R$ as follows:
\begin{align*}
R
&=\frac{F_kF_{k+1}F_{k+2}}{F_{k+1}^3+3F_k^2F_{k+1}-pF_k^3}+\frac{F_kF_{k+1}F_{k+2}}{pF_{k+1}^3+3F_kF_{k+1}^2+F_k^3}+\frac{pF_k}{F_{k+3}}\\
&<\frac{F_kF_{k+1}F_{k+2}}{F_{k+1}^3+2F_k^2F_{k+1}}+\frac{F_kF_{k+1}F_{k+2}}{p^2F_kF_{k+1}^2+3F_kF_{k+1}^2}+\frac{pF_k}{p^3F_k}\\
&=\frac{F_kF_{k+2}}{F_{k+1}^2+2F_k^2}+\frac{F_{k+2}}{(p^2+3)F_{k+1}}+\frac{1}{p^2}\\
&=\frac{pF_kF_{k+1}+F_k^2}{F_{k+1}^2+2F_k^2}+\frac{F_{k+2}}{(p^2+3)F_{k+1}}+\frac{1}{p^2}\\
&\le\frac{F_{k+1}^2+F_k^2}{F_{k+1}^2+2F_k^2}+\frac{(p+1)F_{k+1}}{(p^2+3)F_{k+1}}+\frac{1}{p^2}\\
&<1+\frac{p+1}{p^2}+\frac{1}{p^2}\\
&=1+\frac{1}{p}+\frac{2}{p^2}.
\end{align*}
Since $p\ge 2$, the required estimate follows:
\begin{equation*}
R<1+\frac{1}{p}+\frac{2}{p^2}\le 1+\frac{1}{2}+\frac{2}{2^2}
=2\le p.\qedhere
\end{equation*}
\end{proof}

\section{Proof of the lower estimate of Theorem \ref{t13}}\label{s4}

We do not have an elegant generalization of Lemma \ref{l22}, but we use its underlying idea in Lemma \ref{l41} below.
We recall that the sequence $(F_i)$ satisfies Binet's formula
\begin{equation*}
F_i=\frac{A^i-B^i}{A-B},\quad i=0,1,\ldots,
\end{equation*}
where
\begin{equation*}
A:=\frac{p+\sqrt{p^2+4q}}{2}\qtq{and}B:=\frac{p-\sqrt{p^2+4q}}{2}=-\frac{q}{A},
\end{equation*}
are the solutions of the equation $x^2=px+q$.
Observe that
\begin{equation*}
p<A<p+1\qtq{and}-1<B<0.
\end{equation*}

\begin{lemma}\label{l41}
Set
\begin{equation*}
a_i:=\frac{A^2-1}{AF_i}-\frac{1}{F_{i-1}}+\frac{1}{F_{i+1}},\quad i=2,3,\ldots .
\end{equation*}
Then the series $\sum_{i=2}^{\infty}a_i$ is of Leibniz type, starting with $a_2>0$.
\end{lemma}

\begin{proof}
Since $F_i\to+\infty$, $a_i\to 0$.
Next we show that
\begin{equation}\label{41}
a_i=-\frac{B^{i-1}H_i} {A^2F_{i-1}F_iF_{i+1}}, \qquad H_i=q(A^2-1)F_{i-1}+pA^2F_i>0. 
\end{equation}
First we take a common denominator in the definition of $a_i$ to get
\begin{equation*}
a_i=\frac{(A^3-A)F_{i-1}F_{i+1}-A^2F_iF_{i+1}+A^2F_{i-1}F_i}{A^2F_{i-1}F_iF_{i+1}}.
\end{equation*}
Hence \eqref{41} will follow if we prove the identities
\begin{align*}
&F_i-AF_{i-1}=B^{i-1}
\intertext{and}
&(A^3-A)F_{i-1}F_{i+1}-A^2F_iF_{i+1}+A^2F_{i-1}F_i
=(AF_{i-1}-F_i)[q\bj{A^2-1}F_{i-1}+pA^2F_i].
\end{align*}
for every $i\ge 1$.
For $i=1,2$ the first identity reduces to the obvious equalities $1=1$ and $p-A=B$; the general case hence follows by induction because $F_{i+1}=pF_i+qF_{i-1}$ and $B^2=pB+q$.

Next, eliminating $F_{i+1}$ from the second identity by using the relation $F_{i+1}=pF_i+qF_{i-1}$, it remains to show that
\begin{equation*}
[(A^3-A)F_{i-1}-A^2F_i]\bj{pF_i+qF_{i-1}}+A^2F_{i-1}F_i
=(AF_{i-1}-F_i)[q\bj{A^2-1}F_{i-1}+pA^2F_i].
\end{equation*}
Developing both sides, after simplification this reduces to the obvious relation $A^2=pA+q$.

It follows from \eqref{41} that the sequence $(a_i)$ has alternating sign.

It remains to show the inequalities $|a_i|>|a_{i+1}|$.
For this we observe that
\begin{align*}
H_iF_{i+2}&-H_{i+1}F_{i-1}\\
&=[q\bj{A^2-1}F_{i-1}+pA^2F_i]\bj{pF_{i+1}+qF_i}-[q\bj{A^2-1}F_i+pA^2F_{i+1}]F_{i-1}\\
&= q(A^2-1) F_{i-1} (F_{i+2}-F_i)
  + p A^2 [(pF_i-F_{i-1})F_{i+1} + qF_i^2]>0
\end{align*}
because $A>p\ge q\ge 1$, $F_{i+2}>F_i$, $F_{i-1}>0$, $pF_i-F_{i-1}\ge 0$, and $qF_i^2>0$.
The monotonicity hence follows:
\begin{equation*}
\abs{\frac{a_{i+1}}{a_i}}
=\frac{|B|^iH_{i+1}}{A^2F_iF_{i+1}F_{i+2}}
\cdot\frac{A^2F_{i-1}F_iF_{i+1}}{|B|^{i-1}H_i}
=\frac{H_{i+1}F_{i-1}}{H_iF_{i+2}}|B|<|B|<1.\qedhere
\end{equation*}
\end{proof}

\begin{lemma}
Let $k\ge 1$ be an odd integer.
If 
\begin{equation}\label{42}
-\frac{1}{F_k}+\bj{p+\frac{q-1}{A+1}}\frac{1}{F_{k+1}}+\frac{A+q}{A+1}\frac{1}{F_{k+2}}
\le\frac{A+q}{A+1}\bj{\frac{1}{F_K}+\frac{1}{F_{K+1}}}
\end{equation}
for $K:=K_1(k)$, then the first estimate of Theorem \ref{t13} holds.
\end{lemma}

\begin{remark*}
The relation \eqref{42} is equivalent to
\begin{equation*}
-\frac{A+1}{F_k}+\frac{A^2+p-1}{F_{k+1}}+\frac{A+q}{F_{k+2}}\le\frac{A+q}{F_K}+\frac{A+q}{F_{K+1}}.
\end{equation*}
\end{remark*}

\begin{proof}
We recall from Remark \ref{r16} (ii) that $K_1(k)\ge k+1$.
For $K=k+1$ the inequality \eqref{42} reduces to $-\frac{1}{F_k}+\frac{p-1}{F_{k+1}}\le 0$, which is satisfied because $F_{k+1}=pF_k+qF_{k-1}>(p-1)F_k$.

Henceforth we assume that $K\ge k+2$.
Applying Lemma \ref{l41} we have the following relations:
\begin{equation}\label{43}
\begin{split}
\frac{p-1}{F_{k+1}}&+\bj{\sum_{i=k+2}^{K}\frac{p+q-1}{F_i}}-\frac{1}{F_k}\\
&=-\frac{1}{F_k}-\frac{q}{F_{k+1}}+\bj{\sum_{i=k+1}^{K}\frac{p+q-1}{F_i}}\\
&=-\frac{1}{F_k}-\frac{q}{F_{k+1}}+\frac{(p+q-1)A}{A^2-1}\sum_{i=k+1}^{K}\bj{\frac{1}{F_{i-1}}-\frac{1}{F_{i+1}}+a_i}\\
&=-\frac{1}{F_k}-\frac{q}{F_{k+1}}
+\frac{(p+q-1)A}{A^2-1}\bj{\frac{1}{F_k}+\frac{1}{F_{k+1}}-\frac{1}{F_K}-\frac{1}{F_{K+1}}+\sum_{i=k+1}^{K}a_i}
.
\end{split}
\end{equation}
If $k$ is odd, then the last Leibniz type alternating sum starts with a positive term, so that
\begin{align*}
\frac{p-1}{F_{k+1}}&+\bj{\sum_{i=k+2}^{K}\frac{p+q-1}{F_i}}-\frac{1}{F_k}\\
&<-\frac{1}{F_k}-\frac{q}{F_{k+1}}
+\frac{(p+q-1)A}{A^2-1}\bj{\frac{1}{F_k}+\frac{1}{F_{k+1}}-\frac{1}{F_K}-\frac{1}{F_{K+1}}+a_{k+1}}\\
&=-\frac{1}{F_k}-\frac{q}{F_{k+1}}
+\frac{(p+q-1)A}{A^2-1}\bj{\frac{1}{F_{k+1}}+\frac{1}{F_{k+2}}+\frac{A^2-1}{AF_{k+1}}-\frac{1}{F_K}-\frac{1}{F_{K+1}}}\\
&=-\frac{1}{F_k}+\bj{\frac{(p+q-1)A}{A^2-1}+p-1}\frac{1}{F_{k+1}}+\frac{(p+q-1)A}{A^2-1}\bj{\frac{1}{F_{k+2}}-\frac{1}{F_K}-\frac{1}{F_{K+1}}}.
\end{align*}
Since
\begin{equation*}
\frac{(p+q-1)A}{A^2-1}
=\frac{(A^2-q)+(q-1)A}{A^2-1}
=\frac{(A^2-1)+(q-1)(A-1)}{A^2-1}
=1+\frac{q-1}{A+1}
=\frac{A+q}{A+1},
\end{equation*}
we conclude that
\begin{multline*}
\frac{p-1}{F_{k+1}}+\bj{\sum_{i=k+2}^{K}\frac{p+q-1}{F_i}}-\frac{1}{F_k}\\
<-\frac{1}{F_k}+\bj{p+\frac{q-1}{A+1}}\frac{1}{F_{k+1}}+\frac{A+q}{A+1}\cdot\bj{\frac{1}{F_{k+2}}-\frac{1}{F_K}-\frac{1}{F_{K+1}}}.
\end{multline*}
We conclude by observing that the last expression is $\le 0$ by \eqref{42}.
\end{proof}

Let us denote the left and right hand side of  \eqref{42} by $L_k$ and $R_K$, respectively.
We are going to show that \eqref{42} is satisfied with $K=K_1(k)$ defined in Theorem \ref{t13}.
It is convenient to introduce the number
\begin{equation*}
t:=-\frac{B}{A}=\frac{q}{A^2}\in (0,1).
\end{equation*}
Then
\begin{multline}\label{44}
\frac{1}{F_j}=\frac{A-B}{A^j-B^j}
=\frac{A(1+t)}{A^j(1-(-t)^j)}\\
=A^{1-j}(1+t)\bj{1+(-t)^j+(-t)^{2j}+(-t)^{3j}+(-t)^{4j}+\cdots},\quad j=1,2,\ldots .
\end{multline}

Next we  estimate the two sides of the inequality \eqref{42}.

\begin{lemma}\label{l43}
The left hand side $L_k$ of  \eqref{42} satisfies the following inequalities:
\begin{equation}\label{45}
\frac{L_k}{1+t}<A^{1-k}\bj{-1+t^k}+\bj{p+\frac{q-1}{A+1}}\frac{A^{-k}}{1-t^{k+1}}+\bj{1+\frac{q-1}{A+1}}A^{-k-1}.
\end{equation}
\end{lemma}

\begin{remark*}
The inequality \eqref{45} is equivalent to
\begin{equation*}
\frac{L_k}{1+t}<A^{1-k}\bj{-1+t^k}+\frac{A^2+p-1}{A+1}\cdot\frac{A^{-k}}{1-t^{k+1}}+\frac{A+q}{A+1}\cdot A^{-k-1}.
\end{equation*}
\end{remark*}

\begin{proof}
Using \eqref{44} we may rewrite the definition of $L_k$ in the following form (we recall that $k$ is odd):
\begin{align*}
\frac{L_k}{1+t}=A^{1-k}&\bj{-1+t^k-t^{2k}+t^{3k}-t^{4k}+\cdots}\\
&+\bj{p+\frac{q-1}{A+1}}A^{-k}\bj{1+t^{k+1}+t^{2k+2}+t^{3k+3}+t^{4k+4}+\cdots }\\
&+\bj{1+\frac{q-1}{A+1}}A^{-k-1}\bj{1-t^{k+2}+t^{2k+4}-t^{3k+6}+t^{4k+8}+\cdots}.
\end{align*}
Since $0<t<1$, by the classical majoration of Leibniz type alternating sums we have
\begin{align*}
&-1+t^k-t^{2k}+t^{3k}-t^{4k}+\cdots<-1+t^k
\intertext{and}
&1-t^{k+2}+t^{2k+4}-t^{3k+6}+t^{4k+8}+\cdots <1.
\intertext{Finally,}
&1+t^{k+1}+t^{2k+2}+t^{3k+3}+t^{4k+4}+\cdots 
=\frac{1}{1-t^{k+1}}.\qedhere
\end{align*}
\end{proof}

\begin{lemma}\label{l44}
Setting
\begin{equation*}
T_K=\frac{A}{1-(-t)^K}+\frac{1}{1-(-t)^{K+1}},
\end{equation*}
the following estimates hold for all $K\ge 2$:
\begin{equation*}
A+1-\frac1{A^2} <T_K <A+1+A^{1-K} <A+2.
\end{equation*}
\end{lemma}

\begin{proof}
If $K$ is even, then
\begin{equation*}
T_K-(A+1)=\frac{A}{1-t^K}+\frac{1}{1+t^{K+1}}-(A+1)
=\frac{At^K}{1-t^K}-\frac{t^{K+1}}{1+t^{K+1}}
>At^K-t^{K+1}>0
\end{equation*}
because $0<t<1<A$, and
\begin{equation*}
T_K-(A+1)=\frac{A}{1-t^K}+\frac{1}{1+t^{K+1}}-(A+1)
=\frac{At^K}{1-t^K}-\frac{t^{K+1}}{1+t^{K+1}}
<\frac{At^K}{1-t^K}
\le A^{1-K}<1.
\end{equation*}
Only the second to last inequality is not obvious.
Since
\begin{equation*}
\frac{At^K}{1-t^K}
\le A^{1-K}
\Longleftrightarrow (1+A^K)t^K\le 1
\Longleftrightarrow (1+A^K)|B|^K\le A^K
\Longleftrightarrow |B|^K+q^K\le A^K,
\end{equation*}
it is sufficient to show the last inequality.
It is satisfied for $K=1$ because $A-|B|=p\ge q$, and the case $K>1$ hence follows by induction because $|B|<q<A$.

If $K$ is odd, then $K\ge 3$.
Since $0<t<1$ and $At=q/A<1$, now we have
\begin{align*}
T_K-(A+1)
&=\frac{A}{1+t^K}+\frac{1}{1-t^{K+1}}-(A+1)\\
&=\frac{t^{K+1}}{1-t^{K+1}}-\frac{At^K}{1+t^K}
>-\frac{At^K}{1+t^K}
>-At^K
\ge -At^3
>-\frac{1}{A^2},
\end{align*}
and
\begin{equation*}
T_K-(A+1)=\frac{t^{K+1}}{1-t^{K+1}}-\frac{At^K}{1+t^K}<\frac{t^{K+1}}{1-t^{K+1}}\le A^{-K-1}<A^{1-K}.
\end{equation*}
Indeed, in the last case we may repeat the computation given for even $K$'s with $K+1$ instead of $K$.
\end{proof}

\begin{lemma}\label{l45}
The right hand side $R_K$ of  \eqref{42} satisfies the following inequality for all $K\ge 2$:
\begin{equation*}
\frac{R_K}{1+t}>\frac{A+q}{A^K}\bj{1-\frac{1}{A^2(A+1)}}.
\end{equation*}
\end{lemma}

\begin{proof}
The lemma follows from Lemma \ref{l44} because
\begin{equation*}
\frac{R_K}{1+t}=\frac{A+q}{A^K(A+1)}T_K
\end{equation*}
by \eqref{44}.
\end{proof}

Now we are ready to prove the first inequality of Theorem \ref{t13}.
In view of the estimates of $L_k$ and $R_K$ in Lemmas \ref{l43} and \ref{l45}, \eqref{42} will follow if we show  that
\begin{equation*}
A^{1-k}\bj{-1+t^k}+\bj{p+\frac{q-1}{A+1}}\frac{A^{-k}}{1-t^{k+1}}+\bj{1+\frac{q-1}{A+1}}A^{-k-1}
<\frac{A+q}{A^K}\bj{1-\frac{1}{A^2(A+1)}},
\end{equation*}
or equivalently that
\begin{equation*}
A^2\bj{-1+t^k}+\bj{p+\frac{q-1}{A+1}}\frac{A}{1-t^{k+1}}+\bj{1+\frac{q-1}{A+1}}
<\frac{(A+q)A^{k+1}}{A^K}\bj{1-\frac{1}{A^2(A+1)}}.
\end{equation*}
Since
\begin{equation*}
-A^2+\bj{p+\frac{q-1}{A+1}}A+\bj{1+\frac{q-1}{A+1}}
=-(pA+q)+pA+1+\frac{q-1}{A+1}(A+1)=0,
\end{equation*}
the relation may be rewritten in the form
\begin{equation*}
A^2t^k+\bj{p+\frac{q-1}{A+1}}\frac{At^{k+1}}{1-t^{k+1}}<\frac{(A+q)A^{k+1}}{A^K}\bj{1-\frac{1}{A^3+A^2}}.
\end{equation*}
Since $k\ge 1$ and $0<t<1$, it is sufficient to prove  the following, slightly stronger estimate:
\begin{equation*}
A^2t^k+\bj{p+\frac{q-1}{A+1}}\frac{At^{k+1}}{1-t^2}<\frac{(A+q)A^{k+1}}{A^K}\bj{1-\frac{1}{A^3+A^2}}.
\end{equation*}
Using the definition $t=qA^{-2}$ we eliminate $t$, and the last inequality takes the following form:
\begin{equation*}
q^kA^{2-2k}+\bj{p+\frac{q-1}{A+1}}\frac{q^{k+1}A^{-2k-1}}{1-q^2A^{-4}}<\frac{(A+q)A^{k+1}}{A^K}\bj{1-\frac{1}{A^3+A^2}}.
\end{equation*}
This is equivalent to
\begin{equation*}
q^kA^{K-3k+1}\left\lbrace 1+\bj{p+\frac{q-1}{A+1}}\frac{qA^{-3}}{1-q^2A^{-4}}\right\rbrace
<(A+q)\bj{1-\frac{1}{A^3+A^2}}
\end{equation*}
or to
\begin{equation*}
q^kA^{K-3k+1}\left\lbrace 1+\bj{p+\frac{q-1}{A+1}}\frac{qA}{A^4-q^2}\right\rbrace
<(A+q)\bj{1-\frac{1}{A^3+A^2}}.
\end{equation*}
Multiplying by $(A^3+A^2)(A^4-q^2)$ it takes the following form:
\begin{equation*}
q^kA^{K-3k+1}\left\lbrace (A^3+A^2)(A^4-q^2)+(pA+p+q-1)qA^3\right\rbrace
<(A+q)(A^4-q^2)\bj{A^3+A^2-1}.
\end{equation*}
By Remark \ref{r14} the lower estimate of Theorem \ref{t13} will be proved if we establish the following inequality:
\begin{equation}\label{46}
(p+q-1)\left\lbrace (A^3+A^2)(A^4-q^2)+(pA+p+q-1)qA^3\right\rbrace
<(A+q)(A^4-q^2)\bj{A^3+A^2-1}.
\end{equation}

The inequality \eqref{46} may be simplified to
\begin{multline*}
(p+q-1)\left\lbrace A^7+A^6+pqA^4+(pq-q)A^3-q^2A^2\right\rbrace\\ <A^8+(q+1)A^7+qA^6-A^5-(q^2+q)A^4-(q^3+q^2)A^3-q^3A^2+q^2A+q^3.
\end{multline*}
We have to show that the difference between the right and left hand sides of this relation is positive, i.e., that
\begin{multline*}
\Delta:=A^8+(2-p)A^7+(1-p)A^6-A^5+(-q^2-q-p^2q-pq^2+pq)A^4\\
+(-q^3-p^2q-pq^2+2pq-q)A^3+(pq^2-q^2)A^2+q^2A+q^3>0.
\end{multline*}
Let us transform the last expression by applying repeatedly the relation $A^2=pA+q$ (for more visibility the terms denoted by ``$\ldots$'' are unchanged):
%
\begin{align*}
\Delta
&=2A^7+(q+1-p)A^6-A^5+(-q^2-q-p^2q-pq^2+pq)A^4\\
&\qquad\qquad +(-q^3-p^2q-pq^2+2pq-q)A^3+(pq^2-q^2)A^2+q^2A+q^3\\
&=(q+1+p)A^6+(2q-1)A^5+(-q^2-q-p^2q-pq^2+pq)A^4\\
&\qquad\qquad +(-q^3-p^2q-pq^2+2pq-q)A^3+(pq^2-q^2)A^2+q^2A+q^3\\
&=(pq+p+p^2+2q-1)A^5+(2pq-p^2q-pq^2)A^4+(-q^3-p^2q-pq^2+2pq-q)A^3\\
&\qquad\qquad +(pq^2-q^2)A^2+q^2A+q^3\\
&=(p^2+p^3+4pq-pq^2-p)A^4+(-q^3+3pq+2q^2-2q)A^3+(pq^2-q^2)A^2+q^2A+q^3\\
&=(p^2+p^3-pq^2-p)A^4+(4p^2q-q^3+3pq+2q^2-2q)A^3+(5pq^2-q^2)A^2+q^2A+q^3\\
&=p\bigl((p^2-q^2)+(p-1)\bigr)A^4+q\bigl(4p^2-q^2+3p+2q-2\bigr)A^3+q^2(5p-1)A^2+q^2A+q^3.
\end{align*}
Since $p\ge q\ge 1$ and $A>0$, the last expression is positive.
Hence $\Delta>0$, and the proof is completed.

\section{Proof of the upper estimate of Theorem \ref{t13}}\label{s5}

We recall from Remark \ref{r16} (ii) that $K_2(k)\ge k+2$.
Setting $K:=K_2(k)$ for brevity, we have to prove for every positive odd integer $k$ the inequality
\begin{equation}\label{51}
\frac{p-1}{F_{k+1}}+\bj{\sum_{i=k+2}^{K}\frac{p+q-1}{F_i}}-\frac{1}{F_k}>0.
\end{equation}

\begin{lemma}\label{l51}
Let $k\ge 1$ be an odd integer.
If 
\begin{equation}\label{52}
-\frac{A+1}{F_k}
+\frac{(p-1)(A+1)}{F_{k+1}}
+\frac{(p+q)A+p+2q-1}{F_{k+2}}
+\frac{A+q}{F_{k+3}}
>\frac{A+q}{F_K}+\frac{A+q}{F_{K+1}},
\end{equation}
then the inequality \eqref{51} holds.
\end{lemma}

\begin{proof}
Let us denote by $L$ the left hand side of \eqref{51}.
By the identity \eqref{43} we have
\begin{equation*}
L=-\frac{1}{F_k}-\frac{q}{F_{k+1}}
+\frac{(p+q-1)A}{A^2-1}\bj{\frac{1}{F_k}+\frac{1}{F_{k+1}}-\frac{1}{F_K}-\frac{1}{F_{K+1}}+\sum_{i=k+1}^{K}a_i}.
\end{equation*}
Since $K\ge k+2$, Lemma \ref{l41} gives the following relations:
\begin{align*}
L
&\ge -\frac{1}{F_k}-\frac{q}{F_{k+1}}
+\frac{(p+q-1)A}{A^2-1}\bj{\frac{1}{F_k}+\frac{1}{F_{k+1}}-\frac{1}{F_K}-\frac{1}{F_{K+1}}+a_{k+1}+a_{k+2}}\\
&=-\frac{1}{F_k}-\frac{q}{F_{k+1}}
+\frac{(p+q-1)A}{A^2-1}\bj{\frac{1}{F_{k+2}}+\frac{1}{F_{k+3}}-\frac{1}{F_K}-\frac{1}{F_{K+1}}+\frac{A^2-1}{AF_{k+1}}+\frac{A^2-1}{AF_{k+2}}}\\
&=-\frac{1}{F_k}+\frac{p-1}{F_{k+1}}
+\bj{p+q+\frac{q-1}{A+1}}\frac{1}{F_{k+2}}
+\frac{A+q}{A+1}\bj{\frac{1}{F_{k+3}}-\frac{1}{F_K}-\frac{1}{F_{K+1}}}.
\end{align*}

Hence \eqref{51} is satisfied if the last expression is $\>0$, and this is equivalent to \eqref{52}.
\end{proof}

Now we estimate the left and right hand sides of \eqref{52}.
Similarly to the preceding section, we introduce the number 
\begin{equation*}
t:=-\frac{B}{A}=\frac{q}{A^2}\in (0,1),
\end{equation*}
and we recall from \eqref{44} the identity
\begin{equation}\label{53}
\frac{1}{(1+t)F_j}
=\frac{A}{A^j(1-(-t)^j)}
=A^{1-j}\bj{1+(-t)^j+(-t)^{2j}+(-t)^{3j}+(-t)^{4j}+\cdots}
\end{equation}
for all $j\ge 1$.

\begin{lemma}\label{l52}
The left hand side $L_k$ of \eqref{52} satisfies the following inequality:
\begin{multline}\label{54}
L_k>(1+t)q^kA^{-3k-2}\left[
\bj{p^3+p^2+3pq}A+\bj{p^2q+2pq+q^2-q}-(p+q)q^2A^{-2}
\right.\\
\left.
-(p+2q-1)q^2A^{-3}-q^k\bj{A^{4-2k}+A^{3-2k}}
\right].
\end{multline}
\end{lemma}

\begin{proof}
Using \eqref{53} we may rewrite $L_k$ in the following form (we recall that $k$ is odd):
\begin{align*}
\frac{L_k}{1+t}=A^{1-k}(A+1)&\bj{-1+t^k-t^{2k}+t^{3k}-t^{4k}+\cdots}\\
&+\bj{p-1}A^{-k}(A+1)\bj{1+t^{k+1}+t^{2k+2}+t^{3k+3}+t^{4k+4}+\cdots }\\
&+\bj{(p+q)A+p+2q-1}A^{-k-1}\bj{1-t^{k+2}+t^{2k+4}-t^{3k+6}+t^{4k+8}+\cdots}\\
&+\bj{A+q}A^{-k-2}\bj{1+t^{k+3}+t^{2k+6}+t^{3k+9}+t^{4k+12}+\cdots }.
\end{align*}
Since $A^2=pA+q$, here we have
\begin{align*}
-A^{1-k}&(A+1)+(p-1)A^{-k}(A+1)+\bj{(p+q)A+p+2q-1}A^{-k-1}+\bj{A+q}A^{-k-2}\\
&=A^{-k-2}\left[-A^4-A^3+(p-1)\bj{A^3+A^2}+(p+q)A^2+\bj{p+2q-1}A+\bj{A+q}\right]\\
&=A^{-k-2}\left[-A^4+(p-2)A^3+(2p+q-1)A^2+(p+2q)A+q\right]\\
&=A^{-k-2}\left[-2A^3+(2p-1)A^2+(p+2q)A+q\right]\\
&=A^{-k-2}\left[-A^2+pA+q\right]\\
&=0.
\end{align*}
Therefore the preceding identity simplifies to
\begin{align*}
\frac{L_k}{1+t}=A^{1-k}(A+1)&\bj{t^k-t^{2k}+t^{3k}-t^{4k}+\cdots}\\
&+(p-1)A^{-k}(A+1)\bj{t^{k+1}+t^{2k+2}+t^{3k+3}+t^{4k+4}+\cdots }\\
&+\left[(p+q)A+p+2q-1\right] A^{-k-1}\bj{-t^{k+2}+t^{2k+4}-t^{3k+6}+t^{4k+8}+\cdots}\\
&+\bj{A+q}A^{-k-2}\bj{t^{k+3}+t^{2k+6}+t^{3k+9}+t^{4k+12}+\cdots }.
\end{align*}
Since $0<t<1$, by the standard minoration of Leibniz type alternating series we infer that
\begin{equation*}
\frac{L_k}{1+t}>A^{1-k}(A+1)\bj{t^k-t^{2k}}+(p-1)A^{-k}(A+1)t^{k+1}-\left[(p+q)A+p+2q-1\right]A^{-k-1}t^{k+2},
\end{equation*}
or equivalently that
\begin{equation*}
\frac{L_kA^{k+2}}{(1+t)t^k}>(A^4+A^3)\bj{1-t^k}+(p-1)(A^3+A^2)t-\left[(p+q)A+p+2q-1\right]At^2.
\end{equation*}
Using the equality $t=qA^{-2}$, and then repeatedly the identity $A^2=pA+q$, we infer the following relations:
\begin{align*}
\frac{L_kA^{3k+2}}{(1+t)q^k}&+q^k\bj{A^{4-2k}+A^{3-2k}}\\
&>A^4+A^3+(p-1)q(A+1)-(p+q)q^2A^{-2}-(p+2q-1)q^2A^{-3}\\
&=(p+1)A^3+qA^2+(p-1)qA+(p-1)q-(p+q)q^2A^{-2}-(p+2q-1)q^2A^{-3}\\
&=\bj{p^2+p+q}A^2+2pqA+(p-1)q-(p+q)q^2A^{-2}-(p+2q-1)q^2A^{-3}\\
&=\bj{p^3+p^2+3pq}A+\bj{p^2q+2pq+q^2-q}-(p+q)q^2A^{-2}-(p+2q-1)q^2A^{-3}.
\end{align*}
This is equivalent to \eqref{54}.
\end{proof}

\begin{lemma}\label{l53}
The right hand side $R_K$ of \eqref{52} satisfies the following inequalities:
\begin{equation*}
R_K<(1+t)A^{-K}(A+q)(A+1+A^{1-K})<(1+t)A^{-K}(A+q)(A+2).
\end{equation*}
\end{lemma}

\begin{proof}
Applying \eqref{53} we obtain that
\begin{equation*}
\frac{R_K}{1+t}
=A^{-K}(A+q)\bj{\frac{A}{1-(-t)^K}+\frac{1}{1-(-t)^{K+1}}}
=:A^{-K}(A+q)T_K
\end{equation*}
We conclude by applying Lemma \ref{l44}.
\end{proof}

We recall that the validity of the upper estimate of Theorem \ref{t13} follows from the estimate \eqref{51}.
In order to prove the latter we first give a technical lemma:

\begin{lemma}\label{l54}
Let $\eps=1$ or $\eps=A^{1-K}$.
If 
\begin{multline}\label{55}
(p+q-1)
\left[
\bj{p^3+p^2+3pq}A+\bj{p^2q+2pq+q^2-q}-(p+q)q^2A^{-2}
\right.\\
\left.
-(p+2q-1)q^2A^{-3}-q^k\bj{A^{4-2k}+A^{3-2k}}
\right]
>q^2A(A+q)(A+1+\eps).
\end{multline}
then the inequality \eqref{51} is satisfied.
\end{lemma}

\begin{proof}
In view of Lemmas \ref{l51}--\ref{l53} it is sufficient to prove the inequality
\begin{multline*}
q^kA^{-3k-2}
\left[
\bj{p^3+p^2+3pq}A+\bj{p^2q+2pq+q^2-q}-(p+q)q^2A^{-2}
\right.\\
\left.
-(p+2q-1)q^2A^{-3}-q^k\bj{A^{4-2k}+A^{3-2k}}
\right]
>A^{-K}(A+q)(A+1+\eps),
\end{multline*}
which we may also write in the following form:
\begin{multline*}
q^{k+2}A^{K-3k-1}
\left[
\bj{p^3+p^2+3pq}A+\bj{p^2q+2pq+q^2-q}-(p+q)q^2A^{-2}
\right.\\
\left.
-(p+2q-1)q^2A^{-3}-q^k\bj{A^{4-2k}+A^{3-2k}}
\right]
>q^2A(A+q)(A+1+\eps).
\end{multline*}
In view of Remark \ref{r14}, it is enough to prove \eqref{55}.
\end{proof}

\begin{lemma}\label{l55}
The inequality \eqref{55} of Lemma \ref{l54} is equivalent to the following one:
\begin{equation}\label{56}
\begin{split}
&\left[(p^2+pq+2q)(p^2-q^2)+(2-\eps)pq^2-\eps q^3+2p^2q-p^2-3pq\right]A\\
&\qquad +\left[(pq+q^2)(p^2-q^2)+pq^2-\eps q^3+p^2q+2pq^2-3pq-2q^2+q\right]\\
&\qquad -(p+q-1)(p+q)q^2A^{-2}-(p+q-1)(p+2q-1)q^2A^{-3}\\
&\qquad >(p+q-1)q^k\bj{A^{4-2k}+A^{3-2k}}.
\end{split}
\end{equation}
\end{lemma}

\begin{proof}
Since the right hand side of \eqref{55} is equal to
\begin{align*}
q^2A^3&+q^2(q+1+\eps)A^2+(1+\eps)q^3A\\
&=\bj{pq^2+q^3+(1+\eps)q^2}A^2+(2+\eps)q^3A\\
&=\bj{p^2q^2+pq^3+(1+\eps)pq^2+(2+\eps)q^3}A+\bj{pq^3+q^4+(1+\eps)q^3},
\end{align*}
\eqref{55} may be rewritten in the form
\begin{align*}
&\bj{(p+q-1)(p^3+p^2+3pq)-\bj{p^2q^2+pq^3+(1+\eps)pq^2+(2+\eps)q^3}}A\\
&\qquad +\bj{(p+q-1)(p^2q+2pq+q^2-q)-\bj{pq^3+q^4+(1+\eps)q^3}}\\
&\qquad -(p+q-1)(p+q)q^2A^{-2}-(p+q-1)(p+2q-1)q^2A^{-3}\\
&\qquad >(p+q-1)q^k\bj{A^{4-2k}+A^{3-2k}}.
\end{align*} 
Simplifying the last expression we obtain \eqref{56} because
\begin{align*}
(p+q-1)&(p^3+p^2+3pq)-\bj{p^2q^2+pq^3+(1+\eps)pq^2+(2+\eps)q^3}\\
&=p^4+p^3+3p^2q+p^3q+p^2q+3pq^2-p^3-p^2-3pq-p^2q^2-pq^3\\
&\qquad -(1+\eps)pq^2-(2+\eps)q^3\\
&=\bj{p^4+p^3q-p^2q^2-pq^3}+\bj{4p^2q+(2-\eps)pq^2-(2+\eps)q^3}-\bj{p^2+3pq}\\
&=(p^2+pq+2q)(p^2-q^2)+(2-\eps)pq^2-\eps q^3+2p^2q-p^2-3pq
\end{align*}
and
\begin{align*}
(p+q-1)&(p^2q+2pq+q^2-q)-(pq^3+q^4+(1+\eps)q^3)\\
&=p^3q+2p^2q+pq^2-pq+p^2q^2+2pq^2+q^3-q^2\\
&\qquad -p^2q-2pq-q^2+q-pq^3-q^4-(1+\eps)q^3\\
&=(pq+q^2)(p^2-q^2)+pq^2-\eps q^3+p^2q+2pq^2-2q^2+q-3pq.\qedhere
\end{align*}
\end{proof}

Now we are ready to prove the upper estimate of Theorem \ref{t13}.
We recall from Example \ref{e15} that Theorem \ref{t13} holds for $p=q=2$ and $k=1$.
In the remaining cases, in view of Lemmas \ref{l54} and \ref{l55}, it is sufficient to prove the estimate \eqref{56} with $\eps=1$ or with $\eps=A^{1-K}$.
We distinguish the  following three cases:
\begin{equation*}
p>q,\quad p=q\ge 3\qtq{and}p=q=2, k\ge 3.
\end{equation*}

\begin{proof}[Proof of \eqref{56} with $\eps=1$ for all $k\ge 1$ when $p>q$]
Since $p>q$ and $p,q$ are integers, we have $p-q\ge 1$ and $p^2-q^2\ge p+q$, so that the coefficient of $A$ and the following term in \eqref{56} may be minorized as follows:
\begin{align*}
(p^2+pq+2q)&\bj{p^2-q^2}+(p-q)q^2+2p^2q-p^2-3pq\\
&\ge (p^2+pq+2q)(p+q)+q^2+2p^2q-p^2-3pq\\
&=p^3+4p^2q-pq+pq^2+3q^2-p^2,
\end{align*}
and 
\begin{align*}
\bj{pq+q^2}&\bj{p^2-q^2}+(p-q)q^2+p^2q+2pq^2-3pq-2q^2+q\\
&\ge\bj{pq+q^2}(p+q)+q^2+p^2q+2pq^2-3pq-2q^2+q\\
&=2p^2q+4pq^2+q^3-3pq-q^2+q.
\end{align*}
Therefore it is sufficient to show that
\begin{multline*}
\bj{p^3+4p^2q-pq+pq^2+3q^2-p^2}A+\bj{2p^2q+4pq^2+q^3-3pq-q^2+q}\\
>(p+q-1)\left[(p+q)q^2A^{-2}+(p+2q-1)q^2A^{-3}+q^k\bj{A^{4-2k}+A^{3-2k}}\right].
\end{multline*}
Since $0<q<p<A$ and $k\ge 1$, on the right hand side we have
\begin{align*}
(p+q-1)(p+q)q^2A^{-2}&<(p+q-1)(p+q),\\
(p+q-1)(p+2q-1)q^2A^{-3}&<3(p+q-1),\\
(p+q-1)q^k\bj{A^{4-2k}+A^{3-2k}}&\le (p+q-1)q\bj{A^2+A}=(p+q-1)q(pA+q+A)\\
&=\bj{p^2q+pq^2+q^2-q}A+(p+q-1)q^2,
\end{align*}
so that it is sufficient to show the following inequality:
\begin{multline*}
\bj{p^3+4p^2q-pq+pq^2+3q^2-p^2}A+\bj{2p^2q+4pq^2+q^3-3pq-q^2+q}\\
>\bj{p^2q+pq^2+q^2-q}A+(p+q-1)\bj{q^2+p+q+3}.
\end{multline*}
This is equivalent to
\begin{equation*}
\bj{p^3+3p^2q-pq+2q^2-p^2+q}A+\bj{2p^2q+4pq^2+q^3-3pq-q^2+q}>(p+q-1)(p+q+3).
\end{equation*}
Since $A>p$, it is sufficient to show that
\begin{equation*}
\bj{p^3+3p^2q-pq+2q^2-p^2+q}p+\bj{2p^2q+4pq^2+q^3-3pq-q^2+q}>(p+q-1)(p+q+3),
\end{equation*}
and this is equivalent by a direct computation to the following inequality:
\begin{equation*}
p^4+3p^3q+5p^2q+pq^2+2p+3q+3>p^3+p^2+4pq+q^2.
\end{equation*}
We complete the proof by observing that all terms are positive on both sides, and
\begin{equation*}
p^4>p^3,\quad 3p^3q>p^2,\quad 5p^2q>4pq\qtq{and} pq^2>q^2.\qedhere
\end{equation*}
\end{proof}

Now we turn to the case $p=q$.
Then the inequality \eqref{56}  of Lemma \ref{l55} reduces to
\begin{multline}\label{57}
[(4-2\eps)p^3-4p^2]A+(4-\eps)p^3-5p^2+p\\
>(2p-1)\left[2p^3A^{-2}+(3p^3-p^2)A^{-3}+p^k\bj{A^{4-2k}+A^{3-2k}}\right],
\end{multline}
so that the validity of \eqref{57} implies the validity of the upper estimate in Theorem \ref{t13} for the corresponding values of $p=q$ and $k$.

\begin{proof}[Proof of \eqref{57} with $\eps=A^{1-K}$ for all $k\ge 1$ when $p=q\ge 3$]
Since $A>p>0$ and $3p^3-p^2>0$, for $k\ge 1$ we have
\begin{equation*}
2p^3A^{-2}+(3p-1)p^2A^{-3}+p^k\bj{A^{4-2k}+A^{3-2k}}\le 2p+3+pA^2+pA=p^2+2p+3+(p^2+p)A.
\end{equation*}
Therefore, since $2p-1>0$, it suffices to prove that
\begin{equation*}
[(4-2\eps)p^3-4p^2]A+(4-\eps)p^3-5p^2+p
>(2p-1)\bj{p^2+2p+3+(p^2+p)A}.
\end{equation*}
This is equivalent to
\begin{equation*}
[2p^3-5p^2+p]A+[2p^3-8p^2-3p+3]>\eps p^3(2A+1).
\end{equation*}
Since $K\ge 3$ by Remark \ref{r16} (ii), $\eps\le A^{-2}$, and it is sufficient to prove the following, stronger inequality:
\begin{equation*}
[2p^3-5p^2+p]A+[2p^3-8p^2-3p+3]>2p^3A^{-1}+p^3A^{-2},
\end{equation*}
which is equivalent to
\begin{equation*}
[2p^3-5p^2+p]A^3+[2p^3-8p^2-3p+3]A^2-2p^3A-p^3>0.
\end{equation*}
We prove this inequality by using the identities
\begin{equation*}
A^2=pA+p\qtq{and}A^3=p(pA+p)+pA=(p^2+p)A+p^2
\end{equation*}
and our assumption $p\ge 3$ as follows:
\begin{align*}
[2p^3-5p^2+p]A^3&+[2p^3-8p^2-3p+3]A^2-2p^3A-p^3\\
&=[2p^3-5p^2+p][(p^2+p)A+p^2]+[2p^3-8p^2-3p+3](pA+p)-2p^3A-p^3\\
&=[2p^5-p^4-14p^3-2p^2+3p]A+[2p^5-3p^4-8p^3-3p^2+3p]\\
&=p^5A\bj{2-\frac{1}{p}-\frac{14}{p^2}-\frac{2}{p^3}+\frac{3}{p^4}}+p^5\bj{2-\frac{3}{p}-\frac{8}{p^2}-\frac{3}{p^3}+\frac{3}{p^4}}\\
&>p^5A\bj{2-\frac{1}{3}-\frac{14}{3^2}-\frac{2}{3^3}}+p^5\bj{2-\frac{3}{3}-\frac{8}{3^2}-\frac{3}{3^3}}\\
&=\frac{p^5A}{27}>0.\qedhere
\end{align*}
\end{proof}

\begin{proof}[Proof of \eqref{57} with $\eps=A^{1-K}$ for all $k\ge 3$ when $p=q=2$]
For $p=2$ inequality \eqref{57} becomes
\begin{equation*}
(16-16\eps)A+(14-8\eps)>48A^{-2}+60A^{-3}+3\cdot 2^k\bj{A^{4-2k}+A^{3-2k}}.
\end{equation*}
Since $2/A<1$, the right hand side is decreasing as $k$ increases.
Therefore, in view of Example \ref{e15}, it is sufficient to prove the following inequality, with $k$ replaced  by $3$.
\begin{equation*}
(16-16\eps)A+(14-8\eps)>48A^{-2}+60A^{-3}+24\bj{A^{-2}+A^{-3}}.
\end{equation*}
Since $\eps\le A^{-2}$, it is sufficient to show that
\begin{equation*}
16A-16A^{-1}+14-8A^{-2}>48A^{-2}+60A^{-3}+24\bj{A^{-2}+A^{-3}}.
\end{equation*}
Multiplying by $A^3/2$ and simplifying this may be written in the following equivalent form:
\begin{equation*}
8A^4+7A^3-8A^2-40A-42>0.
\end{equation*}
The last inequality may be proved directly, by using repeatedly the identity $A^2=2A+2$ as follows:
\begin{align*}
8A^4+7A^3-8A^2-40A-42 = 
&=23A^3+8A^2-40A-42\\
&=54A^2+6A-42\\
&=114A+66>0.
\qedhere
\end{align*}
\end{proof}

\end{document}